\documentclass[12pt,reqno]{amsart}
\usepackage[notref,notcite]{}
\usepackage{graphicx}
\usepackage[english]{babel}
\usepackage{amssymb}
\usepackage{amsmath}
\usepackage{amsmath,color}
\usepackage{wrapfig}\usepackage{comment}
\usepackage{hyperref}

\newcommand{\R}{\mathbb{R}}
\newcommand{\D}{\mathbb{D}}
\newcommand{\Z}{\mathbb{Z}}

\begin{document}

\title[Numerical solution of the two-dimensional Calder\'{o}n problem]{Numerical solution of the two-dimensional Calder\'{o}n problem for domains close to a disk}
\author{Vladimir A. Sharafutdinov and Constantin V. Storozhuk}
\address{Sobolev Institute of mathematics. 4 Koptyug av., Novosibirsk, 630090, Russia}
\email{sharafut@list.ru}
\address{Sobolev Institute of mathematics. 4 Koptyug av., Novosibirsk, 630090, Russia}
\email{stork@math.nsc.ru}
\thanks{The work of the first author was performed according to the Government research assignment for IM SB RAS, project FWNF--2026-0026.}

\setcounter{section}{0}
\setcounter{page}{1}
\newtheorem{theorem}{Theorem}[section]
\newtheorem{lemma}[theorem]{Lemma}
\newtheorem{problem}[theorem]{Problem}
\newtheorem{proposition}[theorem]{Proposition}
\newtheorem{corollary}[theorem]{Corollary}
\newtheorem{conjecture}[theorem]{Conjecture}
\newtheorem{definition}[theorem]{Definition}
\newtheorem{remark}[theorem]{Remark}
\numberwithin{equation}{section}
\setlength{\arraycolsep}{3pt}

\maketitle

\begin{abstract}
For a compact Riemannian surface
$(M,g)$ with non-empty boundary $\Gamma$, the
Diri\-chl\-et-to-Neumann operator (DtN-map)
$\Lambda_g:C^\infty(\Gamma)\to C^\infty(\Gamma)$ is defined by
$\Lambda_gf=\left.\frac{\partial u}{\partial\nu}\right|_\Gamma$,
where $\nu$ is the unit outer normal vector to the boundary and $u$
is the solution to the Dirichlet problem $\Delta_gu=0,\
u|_\Gamma=f$. The Calder\'{o}n problem consists of recovering a
Riemannian surface from its DtN-map. It is well known that $(M,g)$
is determined by $\Lambda_g$ uniquely up to a conformal equivalence.
We suggest a method for numerical solution of the Calder\'{o}n
problem. The method works well at least for Riemannian surfaces
$(M,g)$ close to $({\D},e)$, where ${\D}=\{(x,y)\mid x^2+y^2\le1\}$
is the unit disk and $e=dx^2+dy^2$ is the Euclidean metric. Our
numerical examples confirm the statement: the DtN-map is very
sensitive to small deviations of the shape of a domain.
\end{abstract}

\section{Introduction}

A connected two-dimensional Riemannian manifold $(M,g)$ with
non-empty boundary $\Gamma=\partial M$ will be called {\it a
Riemannian surface}.
We involve the requirement of smoothness of
$M,\ \Gamma$ and $g$ into the definition of a Riemannian surface;
the term ``smooth'' is used as a synonym of ``$C^\infty$-smooth''.
The Laplace -- Beltrami operator $\Delta_g:C^\infty(M)\to
C^\infty(M)$ is defined by
$\Delta_g=g^{ij}\nabla_{\!i}\nabla_{\!j}$, where
$(g^{ij})=(g_{ij})^{-1}$.

For a compact Riemannian surface, {\it the Dirichlet-to-Neumann
operator} (DtN-map)
\begin{equation}
\Lambda_g:C^\infty(\Gamma)\longrightarrow C^\infty(\Gamma)
                                \label{1.1}
\end{equation}
is defined by
$$
\Lambda_gf=\left.\frac{\partial u}{\partial\nu}\right|_{\Gamma},
$$
where $\nu$ is the unit outer normal vector to the boundary and $u$ is the solution to the Dirichlet problem
$$
\Delta_gu=0\quad\mbox{in}\quad M,\quad u|_{\Gamma}=f.
$$

As well known, $\Lambda_g$ is a first order pseudidifferential operator. Besides that, $\Lambda_g$ is a non-negative self-adjoint operator
with respect to the $L^2$-product
$$
(u,v)_{L^2(\Gamma)}=\int_\Gamma u\overline v\,ds_g\quad(u,v\in C^\infty(\Gamma)),
$$
where $ds_g$ is the arc length of the curve $\Gamma$ in the metric $g$.
The one-dimensional kernel of $\Lambda_g$ consists of constant functions while the range of $\Lambda_g$ coincides with the space
\begin{equation}
C_0^\infty(\Gamma)=\Big\{f\in C^\infty(\Gamma) \mid \int_\Gamma f\,ds_g=0\Big\}
                                \label{1.2}
\end{equation}
of functions with zero mean value.

We now pose the inverse problem. The one-dimensional Riemannian manifold
$(\Gamma,ds_g)$ and operator $\Lambda_g$ are assumed to be known. One has to recover $(M,g)$ from the data
$(\Gamma,ds_g,\Lambda_g)$.
This problem is named {\it the geometric problem of electric impedance tomography} in \cite{Sr2} and {\it the Calder\'{o}n problem} in \cite{Be}.
We use the second shorter name in the present paper.

The condition on connectedness of $M$ is necessary in the
Calder\'{o}n problem. Indeed, otherwise one of connected components
of $M$ can be a manifold with no boundary. The data
$(\Gamma,ds_g,\Lambda_g)$ contain no information about such a
component.

The following non-uniqueness in the Calder\'{o}n problem is obvious.

If $\varphi:M\rightarrow M$ is a diffeomorphism of $M$ onto itself
fixing the boundary, $\varphi|_{\partial M}=Id$, then the metric
$g'=\varphi^*g$ satisfies $ds_{g'}=ds_g$ and
$\Lambda_{g'}=\Lambda_g$. The equality $g'=\varphi^*g$ means that
$\langle v,w\rangle_{g'}=\langle
(d_p\varphi)v,(d_p\varphi)w\rangle_{g}$ for any point $p\in M$ and
any vectors $v$ and $w$ belonging to the tangent space $T_pM$.
Hereafter $\langle \cdot,\cdot\rangle_{g}$ stands for the scalar
product of tangent vectors with respect to the metric $g$ and
$d_p\varphi:T_pM\rightarrow T_{\varphi(p)}M$ is the differential of
$\varphi$ at $p$. Observe that $\varphi:(M,g')\rightarrow (M,g)$ is
an isometry of Riemannian manifolds, therefore the non-uniqueness is
clear from the geometric viewpoint.

The Calder\'{o}n problem possesses one more non-uniqueness. The Laplace -- Beltrami operator on a two-dimensional Riemannian manifold is conformally
invariant in the following sense: $\Delta_{\rho g}=\rho^{-1}\Delta_g$ for a positive function
$\rho\in C^\infty(M)$. If the function satisfies the boundary condition $\rho|_\Gamma=1$, then $\Lambda_{\rho g}=\Lambda_g$.

For a smooth map $\varphi:N\rightarrow N'$ between two manifolds, we
define $\varphi^*:C^\infty(N')\rightarrow C^\infty(N)$ by $\varphi^*
u=u\circ\varphi$. Two ambiguities mentioned above exhaust the
non-uniqueness in the Calder\'{o}n problem.

\begin{theorem} \label{Th1.1}
Let $(M_j,g_j)\ (j=1,2)$ be two compact Riemannian surfaces with boundaries $\Gamma_j$ and let
$\varphi:(\Gamma_1,ds_{g_1})\rightarrow(\Gamma_2,ds_{g_2})$ be an isometry preserving the DtN-map, i.e., such that the following diagram is commutative:
$$
\begin{array}{ccc}
C^\infty(\Gamma_1)&\stackrel{\varphi^*}\longleftarrow&C^\infty(\Gamma_2)\\
\Lambda_{g_1}\downarrow&&\downarrow\Lambda_{g_2}\\
C^\infty(\Gamma_1)&\stackrel{\varphi^*}\longleftarrow&C^\infty(\Gamma_2).
\end{array}
$$
Then $\varphi$ extends to a diffeomorphism $\psi:M_1\rightarrow M_2$ such that $\psi|_{\Gamma_1}=\varphi$ and $\psi^*g_2=\rho g_1$ for some
function $0<\rho\in C^\infty(M_1)$ satisfying $\rho|_{\Gamma_1}=1$.
\end{theorem}

There exist at least two proofs of the theorem by Lassas -- Uhlmann
\cite{LaU} and by Belishev \cite{Be}, see also \cite{Sar}. Proofs
presented in \cite{LaU} and \cite{Be} are very different and not
easy. But all difficulties disappear in the case of simply connected
surfaces. In the latter case an elementary proof of Theorem
\ref{Th1.1} is presented in \cite{Sr2}. The proof is based on the
following almost obvious fact: for a compact  Riemannian surface
$(M,g)$, there exists a function $0<\rho\in C^\infty(M)$ such that
$\rho|_\Gamma=1$ and $\rho g$ is a flat metric, i.e., its Gaussian
curvature is identically equal to zero. This implies, in the case of
a simply connected $M$, that $(M,\rho g)$ can be isometrically
immersed into the Euclidean plane. This reduces Theorem \ref{Th1.1}
to the partial case when both surfaces $M_1$ and $M_2$ are simply
connected, probably multi-sheet, planar domains and metrics $g_1$
and $g_2$ coincide with the standard Euclidean metric of ${\R}^2$.
In the latter case, the theorem is easily proved by using basic
properties of conformal maps.

The rest of the article is arranged as follows. Flat Riemannian surfaces are shortly discussed in Section 2.
In Section 3, we demonstrate that, for simply connected Riemannian surfaces, the Calder\'{o}n problem is equivalent to the problem of
recovering a simply
connected domain $\Omega\subset{\R}^2$, furnished by the Euclidean metric, from its DtN-map $\Lambda_\Omega$ (Problem \ref{Pr3.1} below).
In Section 4, we introduce the space ${\mathcal A}(\gamma)$ of boundary traces of holomorphic functions and distinguish the subspace
${\mathcal A}_s(\gamma)\subset{\mathcal A}(\gamma)$ consisting of functions that represent simple curves (i.e., curves with no
self-intersection). Then we prove the uniqueness theorem (Theorem \ref{Th4.1} below) which serves as a base of our reconstruction algorithm.

Sections 2--4 actually reproduce arguments of \cite{Sr2} and \cite{Sar} with minor modifications oriented to the numerical solution of the
Calder\'{o}n problem.
This is true for Theorem \ref{Th4.1} and Proposition \ref{P4.1} as well. Nevertheless, we present proofs of two latter statements
because of their crucial role.

Our reconstruction algorithm is described in Section 5. Let
$\gamma(s)$ be the arc length parametrization of the boundary curve
$\Gamma$ of an unknown domain $\Omega$. Then $\gamma(s)$ solves an
obvious first order differential equation (Equation \eqref{5.1}
below). A less obvious zero order pseudodifferental equation for
$\gamma(s)$ (Equation \eqref{5.2} below) is equivalent to the
statement $\gamma\in {\mathcal A}(\gamma)$. We emphasize that the
operator ${\mathcal H}_\gamma$ participating in \eqref{5.2} is known
since it is easily expressed through the DtN-map. We represent the
unknown function $\gamma$ by its Fourier series, the system
\eqref{5.1}--\eqref{5.2} is replaced by an infinite system of
quadratic equations in Fourier coefficients $\widehat\gamma_n$. Of
course practically we replace the Fourier series by its partial sum
and solve a finite system of quadratic algebraic equation. The
latter system is solved by some iteration method (a version of the
classical gradient descent method). As in any iteration method, the
choice of an initial approximation is of a crucial importance. The
system \eqref{5.1}--\eqref{5.2} has many solutions belonging to
${\mathcal A}(\gamma)$, but only one of them belongs to
${\mathcal A}_s(\gamma)$ by Theorem \ref{Th4.1}. If the unknown domain $\Omega$
is sufficiently close to the unit disk, we can choose an initial
approximation $\gamma^0(s)$ close to $e^{is}$. This is the main reason why we reconstruct only domains close to a disk.

Several numerical examples are presented in the final Section 6. In our opinion, the following statement is the main thesis of the article:

{\it The Dirichlet-to-Neumann operator is very sensitive to small deviations of the shape of a domain.}

Most probably, our method can be also applied for reconstruction of more general simply connected domains $\Omega$ that are not close to a disk if
we additionally know a sufficiently good approximation $\widetilde\Omega$ of $\Omega$, in order to choose an appropriate initial approximation
$\gamma_0$ for our iteration algorithm. So far we have no such an example because of the following. In the current work, the forward problem
(computing
the matrix of the DtN-map for a given domain) is solved by some method applicable for domains close to the disk only, see details
at the beginning of Section 6.
We intend to improve the latter method and to reconstruct general simply connected domains in our forthcoming work.

\section{Flat metrics}

A Riemannian metric on a two-dimensional manifold is said to be {\it a flat metric} if its Gaussian curvature is identically equal to zero.

\begin{proposition} \label{P2.1}
Let $(M,g)$ be a compact Riemannian surface with boundary $\Gamma$.
There exists a unique positive function $\rho\in C^\infty(M)$
satisfying the boundary condition $\rho|_\Gamma=1$ and such that
$\rho g$ is a flat metric.
\end{proposition}

For the proof see \cite[Lemma 2.1]{Sr2} and \cite[Proposition 3.1]{Sar}.

It is well known (and can be easily proved) that a flat  Riemannian
surface $(M,g)$ is locally isometric to the Euclidean plane, i.e.,
for every point $a\in M$, there exist a neighborhood $U_a\subset M$
and smooth injective map
\begin{equation}
i_a:U_a\to{\mathbb C}={\R}^2
                                \label{2.1}
\end{equation}
such that $i_a^*e=g|_{U_a}$ for the standard Euclidean metric
$e=|dz|^2=dx^2+dy^2$ of the plane. If additionally the
surface $M$ is oriented and the map \eqref{2.1} is assumed to
transform the chosen orientation of $M$ to the standard orientation
of ${\R}^2$, then the map $i_a$ is unique up to the composition with
a shift and rotation of the plane. Moreover, the local isometry
\eqref{2.1} uniquely extends along every curve starting at $a$. More
precisely, the latter statement means the following. Let
$\gamma:[0,1]\to M$ be a continuous curve satisfying $\gamma(0)=a$.
There exists a finite sequence $0=t_0<t_1<\dots<t_n=1$ and, for
every point $a_k=\gamma(t_k)$, there exists a local isometry
\begin{equation}
i_{a_k}:U_{a_k}\to{\R}^2\quad(k=0,1,\dots,n)
                                \label{2.2}
\end{equation}
such that (a) $\gamma[0,1]\subset\cup_{k=0}^nU_{a_k}$, (b) $U_{a_k}\cap U_{a_{k+1}}\neq\emptyset$, (c) for $k=0$, the map \eqref{2.2}
coincides with \eqref{2.1}, and (d) the maps $i_{a_k}$, and $i_{a_{k+1}}$ coincide on $U_{a_k}\cap U_{a_{k+1}}$. Under conditions (a)--(d)
all maps $i_{a_k}$ are uniquely determined by the map \eqref{2.1}.

Boundary points give no difficulty to the construction of the previous paragraph. Some points of the sequence $(a_0,\dots a_n)$ can belong
to $\partial M$ or even the inclusion $\gamma[0,1]\subset\partial M$ can hold.

Recall {\it the monodromy principle} from Complex Analysis. If $\Omega\subset{\mathbb C}$ is a simply connected domain, a holomorphic function $f$ is defined in some neighborhood of a point $a\in\Omega$ and admits a holomorphic continuation along every curve $\gamma:[0,1]\to \Omega$ starting at $a$, then $f$ uniquely extends to a holomorphic function on the whole of $\Omega$. We emphasize the importance of the condition: $\Omega$ is a simply connected domain. Analytic continuation in multi-connected domains can result multi-valued analytic functions. The same monodromy principle is valid for flat metrics with the essential simplification: we do not need to assume the existence of continuation along curves, such a continuation always exists. In this way we arrive to the following statement.

\begin{proposition} \label{P2.2}
A compact simply connected flat Riemannian surface $(M,g)$ admits a
locally isometric immersion into the Euclidean plane, i.e., there
exists a smooth map
\begin{equation}
I:M\to{\R}^2
                                \label{2.3}
\end{equation}
such that, for every point $x\in M$, the differential $d_xI:T_xM\to{\R}^2$ is an isometry of the tangent space $T_xM$ furnished with the dot product
$\langle\cdot,\cdot\rangle_g$ onto ${\R}^2$ furnished with the standard Euclidean dot product.
If additionally the surface $M$ is oriented and the map \eqref{2.3} is assumed to transform the chosen orientation of $M$ to the standard
orientation of ${\R}^2$, then the immersion $I$ is unique up to the composition with a shift and rotation of the plane.
\end{proposition}

See details of the proof of Proposition \ref{P2.2} in \cite{Sr2}.

In the simplest case the immersion \eqref{2.3} is an isometry between $(M,g)$
and $(\Omega,e)=(I(M),e)$ ($e$ is the standard Euclidean
metric on ${\R}^2$). In the general case, the immersion \eqref{2.3} is not injective, the closed curve $\Gamma=I(\partial M)$
can have self-intersections, and $\Omega=I(M)$ is {\it a simply connected multi-sheet domain}. See \cite{JS0} for the definition of
a simply connected multi-sheet domain.

In what follows we study the Calder\'{o}n problem for a compact simply connected Riemannian surfaces.
Moreover, we assume the immersion \eqref{2.3} to be injective. In such a case $(M,g)$ can be identified with the closed planar domain
$\Omega=I(M)\subset{\mathbb C}$ bounded by the simple (i.e., with no self-intersection) smooth closed curve $\Gamma=I(\partial M)$.
All our results are actually valid for simply connected multi-sheet domains, but some definitions below should be slightly modified in the latter case.

\section{The Calder\'{o}n problem for a simply connected planar domain}

Let $\Omega\subset{\mathbb C}={\R}^2$ be a compact simply connected
domain  bounded by a simple (i.e., with no self-intersection)
smooth curve $\Gamma$. The domain $\Omega$ is considered as a flat
Riemannian surface with the standard Euclidean metric
$|dz|^2=dx^2+dy^2$. Let
\begin{equation}
\Lambda_\Omega:C^\infty(\Gamma)\to C^\infty(\Gamma)
                                \label{3.1}
\end{equation}
be the DtN-map of the domain $\Omega$. Thus, the notation $\Lambda_\Omega$ is used instead of $\Lambda_g$, see \eqref{1.1}.
We orient the curve $\Gamma$ choosing the counter clockwise direction to be the positive one.

Define the operator $D=D_s:C^\infty(\Gamma)\to C^\infty(\Gamma)$ by $D=-i\frac{d}{ds}$, where $\frac{d}{ds}$ is the differentiation with respect
to the arc length in the positive direction and $i$ is the imaginary unit.
The restriction $D|_{C_0^\infty(\Gamma)}:C_0^\infty(\Gamma)\to C_0^\infty(\Gamma)$ of $D$ to the subspace $C_0^\infty(\Gamma)$ (see \eqref{1.2})
is an isomorphism. Let
\begin{equation}
D^{-1}:C_0^\infty(\Gamma)\to C_0^\infty(\Gamma)
                                \label{3.2}
\end{equation}
be the inverse of $D|_{C_0^\infty(\Gamma)}$. Thus, for $f\in C_0^\infty(\Gamma)$, $D^{-1}f$ is the unique anti-derivative of $f$ with the zero mean
value.

We consider the inverse problem of recovering the domain $\Omega$ from its DtN-map.
The definition \eqref{3.1} should be specified since $\Gamma$ is an unknown curve, the length of $\Gamma$ is known only.
Without lost of generality the length of $\Gamma$ can be assumed to be equal to $2\pi$ since the DtN-operator changes in an obvious
way under the transform
$\Omega\mapsto c\,\Omega\ $ for a constant $c>0$.

Let
${\mathbb S}=\{e^{i s}\mid  s\in{\R}\}$ be the unit circle oriented in the counter clockwise direction. For a function $u\in C^\infty({\mathbb S})$,
we write $u(s)$ instead of $u(e^{i s})$.
We parameterize the curve $\Gamma$ by the arc length $s$ measured in the positive direction from an initial point. The parametrization determines
a smooth $2\pi$-periodic function
$\gamma:{\R}\to {\mathbb C}$ satisfying
\begin{equation}
\Big|\frac{d\gamma( s)}{d s}\Big|=1.
                                \label{3.3}
\end{equation}
The same function can be treated as a diffeomorphism $\gamma:{\mathbb S}\to\Gamma$ preserving the arc length and orientation. With the help of
the diffeomorphism, we transfer the operator \eqref{3.1} onto the circle ${\mathbb S}$, i.e., define the operator
\begin{equation}
\Lambda_\gamma=\gamma^* \Lambda_\Omega\gamma^* {}^{-1}:C^\infty({\mathbb S})\to C^\infty({\mathbb S}).
                                \label{3.4}
\end{equation}
We call $ \Lambda_\gamma$ {\it the DtN-map of the domain $\Omega$ in the natural parametrization}.
Our inverse problem is precisely posed as follows.

\begin{problem} \label{Pr3.1}
Let
$\Omega\subset{\mathbb C}$ be a compact simply connected domain  bounded by a simple smooth curve $\Gamma$ of length $2\pi$.
 Given the operator \eqref{3.4}, one has to recover the closed curve
$\Gamma=\gamma({\mathbb S})$ up to a shift and rotation. Equivalently: Given the operator \eqref{3.4}, one has to recover the
$2\pi$-periodic complex-valued function $\gamma(s)$ satisfying \eqref{3.3}.
\end{problem}

Let us find the value of the operator $ \Lambda_\gamma$ on the function $\gamma\in C^\infty({\mathbb S})$. By \eqref{3.4},
\begin{equation}
 \Lambda_\gamma\gamma=(\gamma^* \Lambda_\Omega\gamma^* {}^{-1})\gamma=(\gamma^* \Lambda_\Omega)(Id),
                                \label{3.5}
\end{equation}
where the function $Id\in C^\infty(\Gamma)$ is defined by $Id(z)=z$ for $z\in \Gamma$. Being defined by $Z(z)=z$, the function
$Z\in C^\infty(\Omega)$ is harmonic in $\Omega$ and satisfies $Z|_\Gamma=Id$. By the definition of the DtN-operator,
$$
\Lambda_\Omega(Id)=\frac{\partial Z}{\partial\nu}\Big|_\Gamma=\nu_x+i\nu_y=\nu,
$$
where the unit outer normal $\nu=(\nu_x,\nu_y)$ to $\Gamma$ is identified with the complex number $\nu=\nu_x+i\nu_y$. Together with \eqref{3.5},
this gives
$\Lambda_\gamma\gamma( s)=(\gamma^* \nu)( s)=\nu(\gamma( s))=-i\,\frac{d\gamma}{d s}$.
We have thus proven
\begin{equation}
\Lambda_\gamma\gamma=-i\,\frac{d\gamma}{d s}.
                                \label{3.6}
\end{equation}
The operator $D^{-1}:C_0^\infty({\mathbb S})\to C_0^\infty({\mathbb S})$ is defined quite similarly to \eqref{3.2}. Since $\gamma$ preserves the arc
length, the operator $D^{-1}$ commutes with $\gamma^*$. Therefore,
applying the operator $D^{-1}$ to the equality \eqref{3.6}, we obtain
\begin{equation}
D^{-1} \Lambda_\gamma\gamma=\gamma.
                                \label{3.7}
\end{equation}
The operator
\begin{equation}
{\mathcal H}_\gamma=D^{-1} \Lambda_\gamma:C^\infty({\mathbb S})\to C^\infty({\mathbb S})
                                \label{3.8}
\end{equation}
is well defined since the range of $\Lambda_\gamma$ coincides with the domain
$C_0^\infty({\mathbb S})$ of $D^{-1}$. The zero order pseudodifferential operator ${\mathcal H}_\gamma$
is sometimes
%где так называют?
called {\it the Hilbert transform of the domain $\Omega$ in the natural parametrization}.
In terms of this operator, the equation \eqref{3.7} is written as follows:
\begin{equation}
{\mathcal H}_\gamma\gamma=\gamma.
                                \label{3.9}
\end{equation}

The equation \eqref{3.9} gives the idea for solving Problem \ref{Pr3.1}: we have to look for a solution $\gamma\in C^\infty({\mathbb S})$
to the pseudodifferential equation \eqref{3.7} satisfying \eqref{3.3}.

\section{The uniqueness theorem}

Let $\Omega\subset{\mathbb C}={\R}^2$ be a compact simply connected domain  bounded by a simple smooth curve $\Gamma$ of length $2\pi$.
We denote by ${\mathcal A}(\Omega)$ the space of complex-valued functions
$f\in C^\infty(\Omega)$ such that the restriction of $f$ to the interior $\Omega\setminus\Gamma$ is a holomorphic function.
We also introduce the space
${\mathcal A}(\Gamma)=\{w|_{\Gamma}\mid w\in{\mathcal A}(\Omega)\}$
of {\it boundary traces of holomorphic functions}.
The trace operator
$$
{\mathcal A}(\Omega)\to{\mathcal A}(\Gamma),\quad w\mapsto w|_{\Gamma}
$$
is an isomorphism by the modulus maximum principle for holomorphic functions.

Repeating arguments of the previous section, we fix a diffeomorphism $\gamma:{\mathbb S}\to\Gamma$ preserving the arc length and orientation
(both are oriented counter clockwise). With the help of the diffeomorphism, we transfer the space ${\mathcal A}(\Gamma)$ to the unit circle,
i.e., define
\begin{equation}
{\mathcal A}(\gamma)=\{\gamma^* f\mid f\in{\mathcal A}(\Gamma)\}.
                                \label{4.1}
\end{equation}

Since ${\mathcal A}(\gamma)\subset C^\infty({\mathbb S})$, every function $f\in{\mathcal A}(\gamma)$ can be considered as a smooth
closed parameterized curve ${\mathbb S}\to{\mathbb C},\  s\mapsto f( s)$. Let ${\mathcal A}_s(\gamma)$ be the set of
$f\in{\mathcal A}(\gamma)$
that correspond to curves without self-intersections (simple curves).

Observe that the function $\gamma\in C^\infty({\mathbb S})$ belongs to ${\mathcal A}(\gamma)$. Indeed, in the previous section,
we have defined the function $Id\in C^\infty(\Gamma)$ by $(Id)(z)=z$. This function belongs to ${\mathcal A}(\Gamma)$ since
it is the restriction to $\Gamma$ of the holomorphic function $Z\in C^\infty(\Omega)$ defined by $Z(z)=z$.
Therefore $\gamma=\gamma^*(Id)\in{\mathcal A}(\gamma)$.

We can now present the precise statement on uniqueness of a solution to Problem \ref{Pr3.1}.

\begin{theorem} \label{Th4.1}
Let the  space ${\mathcal A}(\gamma)$ of functions on the unit circle  ${\mathbb S}$ be defined by \eqref{4.1} for some
compact simply connected domain $\Omega\subset{\mathbb C}$ bounded by a smooth simple curve $\Gamma$ of length $2\pi$. Then

{\rm (1)} The equation
\begin{equation}
\Big|\frac{df}{ds}\Big|=1
                                \label{4.2}
\end{equation}
has a solution $f\in{\mathcal A}_s(\gamma)$.

{\rm (2)} If $f_1,f_2\in{\mathcal A}_s(\gamma)$ are two solutions to the equation \eqref{4.2}, then
$f_1( s)=\alpha f_2( s)+c$ for some complex constants $\alpha$ and $c$ satisfying $|\alpha|=1$.

{\rm (3)} If $f\in{\mathcal A}_s(\gamma)$ is a solution to \eqref{4.2}, then the curve $s\mapsto f(s)$ bounds a domain that is
obtained from $\Omega$ by a shift and rotation.
\end{theorem}

\begin{proof}
As mentioned above, the function $\gamma$ belongs to ${\mathcal A}^\infty(\gamma)$. It solves the equation \eqref{4.2} and
belongs to ${\mathcal A}_s(\gamma)$. This proves the first statement.

To prove last two statements of the theorem, it suffices to demonstrate that every solution
$f\in{\mathcal A}_s(\gamma)$ to the equation \eqref{4.2} is expressed through $\gamma$ by
\begin{equation}
f(s)=\alpha\gamma(s)+c,\quad |\alpha|=1.
                                \label{4.3}
\end{equation}

Let $f\in{\mathcal A}_s(\gamma)$ be a solution to the equation \eqref{4.2}. The function $g=\gamma^* {}^{-1}f$ belongs to
${\mathcal A}(\Gamma)$ and solves the equation
\begin{equation}
\Big|\frac{dg}{ds}\Big|=1.
                                \label{4.4}
\end{equation}
By the definition of ${\mathcal A}(\Gamma)$, there exists a function $w\in C^\infty(\Omega)$ holomorphic in
$\Omega\setminus\Gamma$ and satisfying the boundary condition
$
w|_\Gamma =g.
$
We set $\Gamma'=w(\Gamma)$. Then
$$
\Gamma'=g(\Gamma)=(\gamma^* {}^{-1}f)(\Gamma)=(f\circ\gamma^{-1})(\Gamma)=f({\mathbb S}).
$$
Since $f\in{\mathcal A}_s(\gamma)$, $\Gamma'$ is a smooth simple closed curve. By \eqref{4.3}, the map
$ s\mapsto f(s)$ is the parametrization of $\Gamma'$ by the arc length and $\Gamma'(s)$ runs counter clockwise while $s$ increases.
$\Gamma'$ bounds a compact domain $\Omega'$. Applying {\it the argument principle} \cite[Chapter 1, \S5.23]{LS} to the holomorphic function $w$,
we see that $w$ is a biholomorphism of $\Omega$ onto $\Omega'$. By \eqref{4.2} and \eqref{4.4}, the restriction of the biholomorphism $w$
to $\Gamma$ preserves the arc length, i.e., $\big|w'|_\Gamma\big|=1$. It is well known (and easily proved) that a biholomorphism preserving
the boundary arc length is the composition of a shift and rotation, i.e.,
$$
w(z)=\alpha z+c,\quad|\alpha|=1.
$$
Together with the equality $w|_\Gamma=g=f\circ\gamma^{-1}$, this gives
$$
(f\circ\gamma^{-1})(z)=\alpha z+c\quad(z\in\Gamma).
$$
Setting $z=\gamma(s)$ here, we arrive to \eqref{4.3}.
\end{proof}

{\bf Remark.} Theorem \ref{Th4.1} is not true if the hypothesis
$f\in{\mathcal A}_s(\gamma)$ is replaced with $f\in{\mathcal A}(\gamma)$, as the following example shows. Let $\Omega$ coincide
with the unit disk and $\gamma(s)=e^{i s}$. The function
$f(s)=\frac{1}{n}\,e^{ni s}$ with $n>1$ belongs to
${\mathcal A}(\gamma)$ and solves the equation \eqref{4.2} but cannot be represented in the form \eqref{4.3}.

\begin{proposition} \label{P4.1}
Under hypotheses of Theorem \ref{Th4.1}, a function $f\in C_0^\infty({\mathbb S})$ belongs to ${\mathcal A}(\gamma)$ if and only if
it satisfies ${\mathcal H}_\gamma f=f$.
\end{proposition}

\begin{proof}
We will demonstrate that this statement follows from Proposition 6.1
of \cite{Sar}, where the notation ${\mathcal A}^\infty(\Gamma,g)$ is
used instead of our ${\mathcal A}(\gamma)$. The latter proposition
is proved in \cite{Sar} for a general compact  Riemannian surface
$(M,g)$ that is not assumed to be simply connected. In the case of a
simply connected surface, in particular under hypotheses of Theorem
\ref{Th4.1}, Proposition 6.1 of \cite{Sar} states that
\begin{equation}
{\mathcal A}(\gamma)=\{a+{\mathcal H}_\gamma a\mid a\in C^\infty({\mathbb S})\}.
                                \label{4.5}
\end{equation}
Proposition 6.1 of \cite{Sar} states also that, under hypotheses of Theorem \ref{Th4.1},
$$
\Big(1-(\Lambda_\gamma D^{-1})^2)Da=0\quad\mbox{for every}\quad a\in C^\infty({\mathbb S}),
$$
where 1 is the identity operator (see the remark at the end of Section 6 of \cite{Sar}).
We write the latter statement in the form
$$
Da-\Lambda_\gamma D^{-1}\Lambda_\gamma a=0\quad\mbox{for every}\quad a\in C^\infty({\mathbb S}).
$$
Applying the operator $D^{-1}$ to this equality and using \eqref{3.8}, we obtain
$$
D^{-1}Da={\mathcal H}_\gamma^2 a\quad\mbox{for every}\quad a\in C^\infty({\mathbb S}).
$$
Since $D^{-1}Da=a$ for $a\in C_0^\infty({\mathbb S})$, this implies
\begin{equation}
{\mathcal H}_\gamma^2 a=a\quad\mbox{for every}\quad a\in C_0^\infty({\mathbb S}).
                                \label{4.6}
\end{equation}

Let now $f\in C_0^\infty({\mathbb S})$. By \eqref{4.5}, $f$ belongs to ${\mathcal A}(\gamma)$ if and only if it can be represented in the
form
\begin{equation}
f=a+{\mathcal H}_\gamma a
                                \label{4.7}
\end{equation}
for some $a\in C^\infty({\mathbb S})$. Since $f\in C_0^\infty({\mathbb S})$ and ${\mathcal H}_\gamma a\in C_0^\infty({\mathbb S})$, \eqref{4.7}
implies $a\in C_0^\infty({\mathbb S})$ and \eqref{4.6} holds. Applying the operator ${\mathcal H}_\gamma$ to the equality \eqref{4.7}
and using \eqref{4.6}, we obtain
$$
{\mathcal H}_\gamma f={\mathcal H}_\gamma a+{\mathcal H}_\gamma^2 a={\mathcal H}_\gamma a+a=f.
$$
These arguments are invertible.
\end{proof}

\section{Numerical solution of the Calder\'{o}n problem}

For convenience we reproduce equations \eqref{3.3} and \eqref{3.9} here:
\begin{equation}
\Big|\frac{d\gamma( s)}{d s}\Big|=1,
                                \label{5.1}
\end{equation}
\begin{equation}
{\mathcal H}_\gamma\gamma=\gamma.
                                \label{5.2}
\end{equation}

Let again $\Omega\subset{\R}^2$ be a compact simply connected domain  bounded by a simple smooth curve $\Gamma$ of length $2\pi$.
But now $\Omega$ is an unknown domain, we are given the DtN-map $\Lambda_\gamma$ only for a diffeomorhism
$\gamma:{\mathbb S}\to\Gamma$ preserving the arc length. The knowledge of $\Lambda_\gamma$ is equivalent to the knowledge of
${\mathcal H}_\gamma$ since the differentiation $D:C^\infty({\mathbb S})\to C^\infty({\mathbb S})$ is known.
Given the operator ${\mathcal H}_\gamma$, we try to recover the domain $\Omega$ up to a shift and rotation. In other words,
we look for the function $\gamma$ up to the transform
\begin{equation}
\gamma\mapsto\alpha\gamma+c,\quad(|\alpha|=1,c\in{\mathbb C}).
                                \label{5.3}
\end{equation}
The function $\gamma$ solves the system \eqref{5.1}--\eqref{5.2}. By Proposition \ref{P4.1}, every solution to the system belongs to
${\mathcal A}(\gamma)$. Moreover, by Theorem \ref{Th4.1}, the system has a unique solution $\gamma\in{\mathcal A}_s(\gamma)$ up to the transform
\eqref{5.3}.

Without lost of generality we can assume the function $\gamma$ to have the zero mean value by an appropriate choice of the constant $c$
in \eqref{5.3}. We still have the freedom of multiplying $\gamma$ by a constant $\alpha$ satisfying $|\alpha|=1$.

We represent the function $\gamma$ by its Fourier series
\begin{equation}
\gamma(s)=\sum\limits_{n\in\Z}\widehat\gamma_n\,e^{ins}.
                                \label{5.4}
\end{equation}
Since $\gamma\in C^\infty({\mathbb S})$, the Fourier coefficients fast decay, i.e., for any $K>0$,
$|\widehat\gamma_n|=O((1+|n|)^{-K})$ as $|n|\to\infty$.
Since $\gamma$ has the zero mean value, $\widehat\gamma_0=0$. Without lost of generality we can assume that $\widehat\gamma_1\ge0$
by an appropriate choice of the constant $\alpha$ in \eqref{5.3}.

Let us write equations \eqref{5.1}--\eqref{5.2} in terms of Fourier coefficients $\widehat\gamma_n$. Differentiating \eqref{5.4}, we obtain
$$
\frac{d\gamma}{ds}=i\sum\limits_{n\in\Z}n\,\widehat\gamma_n\,e^{ins}.
$$
Substituting this expression into \eqref{5.1}, we see that
the equation \eqref{5.1} is equivalent to the system
\begin{equation}
\sum\limits_{n\in\Z}n^2|\widehat\gamma_n|^2=1,
                                \label{5.5}
\end{equation}
\begin{equation}
\sum\limits_{n\in\Z}n(n+k)\,\widehat\gamma_n\,\overline{\widehat\gamma_{n+k}}=0\quad(k=\pm1,\pm2,\dots).
                                \label{5.6}
\end{equation}

Let  $H_\gamma=(h_{mn})_{m,n=-\infty}^\infty$ be the matrix of the operator ${\mathcal H}_\gamma$ with respect to the trigonometric basis, i.e.,
$$
(\widehat{{\mathcal H}_\gamma f})_m=\sum\limits_{n\in\Z}h_{mn}\widehat f_n\quad\mbox{for}\quad f\in C^\infty({\mathbb S}).
$$
This is equivalent to
\begin{equation}
{\mathcal H}_\gamma\,e^{ims}=\sum\limits_{n\in\Z}h_{nm}\,e^{ins}\quad(m\in\Z).
                                \label{5.7}
\end{equation}
In particular,
$$
(\widehat{{\mathcal H}_\gamma \gamma})_m=\sum\limits_{n\in\Z}h_{mn}\widehat \gamma_n.
$$
Substituting this expression into \eqref{5.2}, we see that
the equation \eqref{5.2} is equivalent to the system
$$
\sum\limits_{n\in\Z}h_{mn}\widehat\gamma_n=\widehat\gamma_m\quad(m\in\Z).
$$

The matrix $H_\gamma=(h_{mn})_{m,n=-\infty}^\infty$ satisfies
$
h_{m0}=h_{0n}=0
$
since constant functions constitute the kernel of ${\mathcal H}_\gamma$ and its range coincides with $C^\infty_0({\mathbb S})$.

Let $(\lambda_{mn})_{m,n=-\infty}^\infty$ be the matrix of the operator $\Lambda_\gamma$ in the trigonometric basis, i.e.,
$$
(\widehat{\Lambda_\gamma f})_m=\sum\limits_{n\in\Z}\lambda_{mn}\widehat f_n\quad\mbox{for}\quad f\in C^\infty({\mathbb S}).
$$
The matrices $(\lambda_{mn})$ and $(h_{mn})$ are expressed through each other by the equality
\begin{equation}
\lambda_{mn}=mh_{mn}
                                \label{5.7a}
\end{equation}
which immediately follows from \eqref{3.8}.
Recall also that
\begin{equation}
\lambda_{mn}=\overline{\lambda_{nm}}
                                \label{5.7b}
\end{equation}
since $\Lambda_\gamma$ is a self-adjoint operator.

In the general case, $(\lambda_{mn})$ and $(h_{mn})$ are complex-valued matrices. But there is an important exception: both the matrices are
real if the curve $\Gamma$ is symmetric with respect to some line $L$ and the arc length $s$ is measured from an initial point belonging to
$\Gamma\cap L$. We omit the proof of this easy statement.

\medskip

Before going further in the general case, let us consider the simplest example when the domain $\Omega$ coincides with the unit disk.
In this case $\gamma:{\mathbb S}\to{\mathbb C}$ is the identical embedding, we write ${\mathcal H}$ and $H$ instead of ${\mathcal H}_\gamma$
and $H_\gamma$. The functions
$\varphi_n=e^{in s}\ (n=,0,\pm1,\pm2,\dots)$ are eigenvectors of ${\mathcal H}$ and $H$ is the diagonal matrix:
\begin{equation}
H=\mbox{diag}\,(\dots,-1,-1,0,1,1,\dots). \label{H_circle}
\end{equation}
Thus, in the simplest case of the unit disk, the space of solutions to the equation \eqref{5.2}  consists of functions
of the form
\begin{equation}
\gamma(s)=\sum\limits_{n=1}^\infty\widehat\gamma_n\,e^{in s}
                                \label{5.8}
\end{equation}
with arbitrary Fourier coefficients.

Let us now address the equation \eqref{5.1} in the simplest case. By \eqref{5.8}, $\widehat\gamma_n=0$ for $n\le0$ and equations
\eqref{5.5}--\eqref{5.6} take the form
\begin{equation}
\sum\limits_{n=1}^\infty n^2|\widehat\gamma_n|^2=1,
                                \label{5.9}
\end{equation}
\begin{equation}
\sum\limits_{n=1}^\infty n(n+k)\,\widehat\gamma_n\,\overline{\widehat\gamma_{n+k}}=0\quad(k\ge1).
                                \label{5.10}
\end{equation}
To describe all solutions to the system \eqref{5.9}--\eqref{5.10}, it suffices to find all non-zero solutions to the system \eqref{5.10} of
homogeneous equations and then to use \eqref{5.9} as a normalization condition.

The system \eqref{5.10} has the obvious solutions
\begin{equation}
    (\widehat\gamma_1,\widehat\gamma_2,\widehat\gamma_3,\widehat\gamma_4,\dots)=(\underbrace{0,\dots,0}_n,1,0,0,\dots)\quad(n=0,1,\dots).
\label{5.135}\end{equation}
They give the solutions
$
\gamma(s)=\frac{1}{n+1}\,e^{i(n+1)s}\quad(n=0,1,\dots)
$
of the system \eqref{5.1}--\eqref{5.2}. Such a solution belongs to ${\mathcal A}_s(\gamma)$ in the case of $n=0$ only.
Observe that $\widehat\gamma_1=0$ for such a solution if $n>0$. But the inequality $\widehat\gamma_1\neq0$ does not
guarantee that a solution of the system \eqref{5.1}--\eqref{5.2} belongs to ${\mathcal A}_s(\gamma)$, as the following example shows.

Fix $\zeta\in\mathbb C$ satisfying $0<|\zeta|<1$ and set
\begin{equation}
\gamma(s)=\frac{1}{|\zeta|}\Big(e^{is}+\frac{1-|\zeta|^2}{\bar\zeta}\,\ln(1-\bar\zeta\,e^{is})\Big).
                                \label{5.11}
\end{equation}
Fourier coefficients of this function are
$$
\widehat\gamma_n= \left\{ \begin{array}{ll}
0 &\mbox{for}\  n\leq 0,\\
|\zeta| &\mbox{for}\   n=1, \\
    -\frac{1}{n}|\zeta|(1-|\zeta|^2)\zeta^{-1}{\bar\zeta}^{n-2}&\mbox{for}\   n>1.
\end{array} \right.
$$
One easily shows that the function \eqref{5.11} solves equations \eqref{5.1}--\eqref{5.2}. The corresponding curve is drawn on Figure
\ref{fig5.1}
for three values of $\zeta$.

\begin{figure}[h]
\begin{minipage}[h]{0.3\linewidth}
\center{\includegraphics[width=0.7\linewidth]{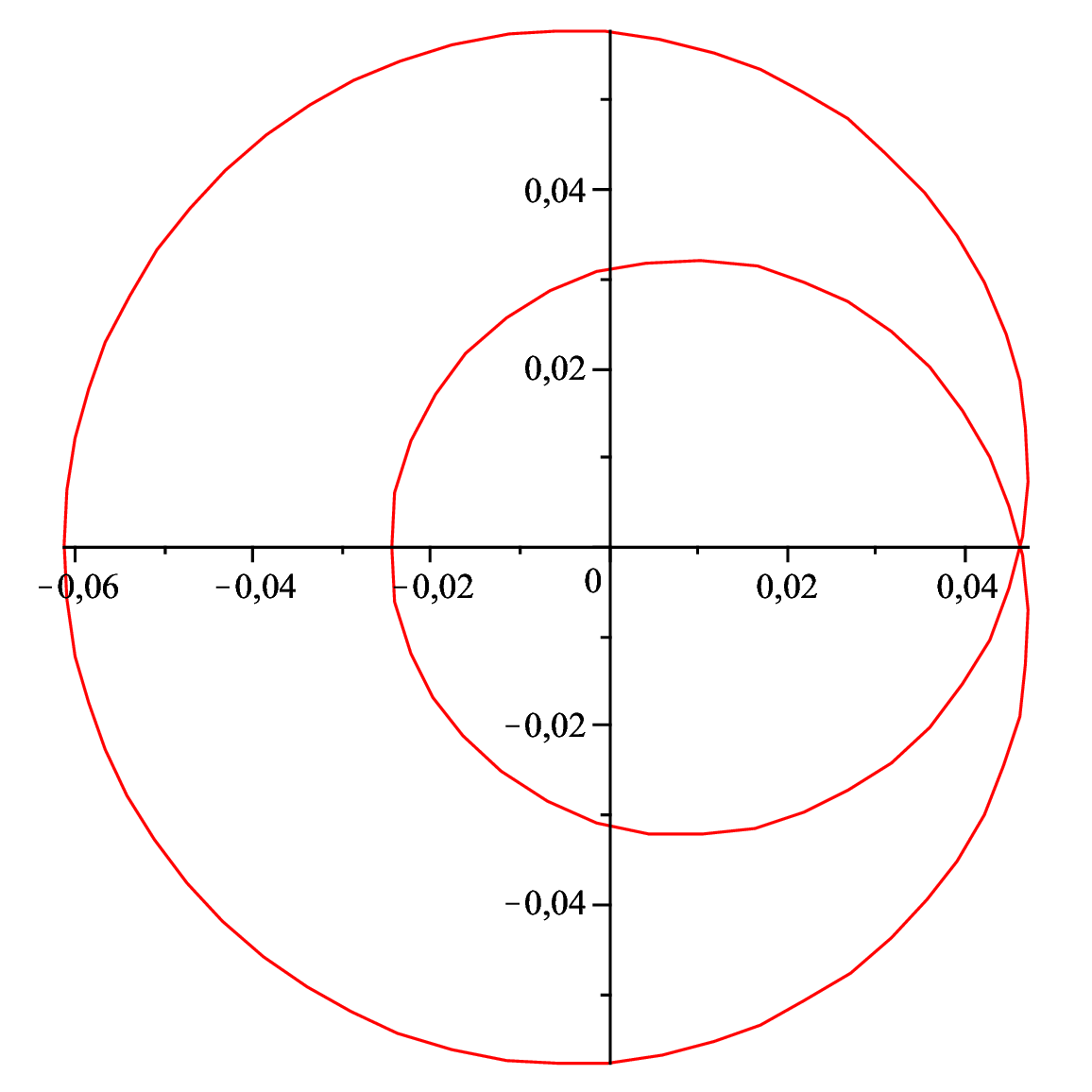}} \\
$\zeta=0.3$
\end{minipage}
\begin{minipage}[h]{0.3\linewidth}
\center{\includegraphics[width=0.7\linewidth]{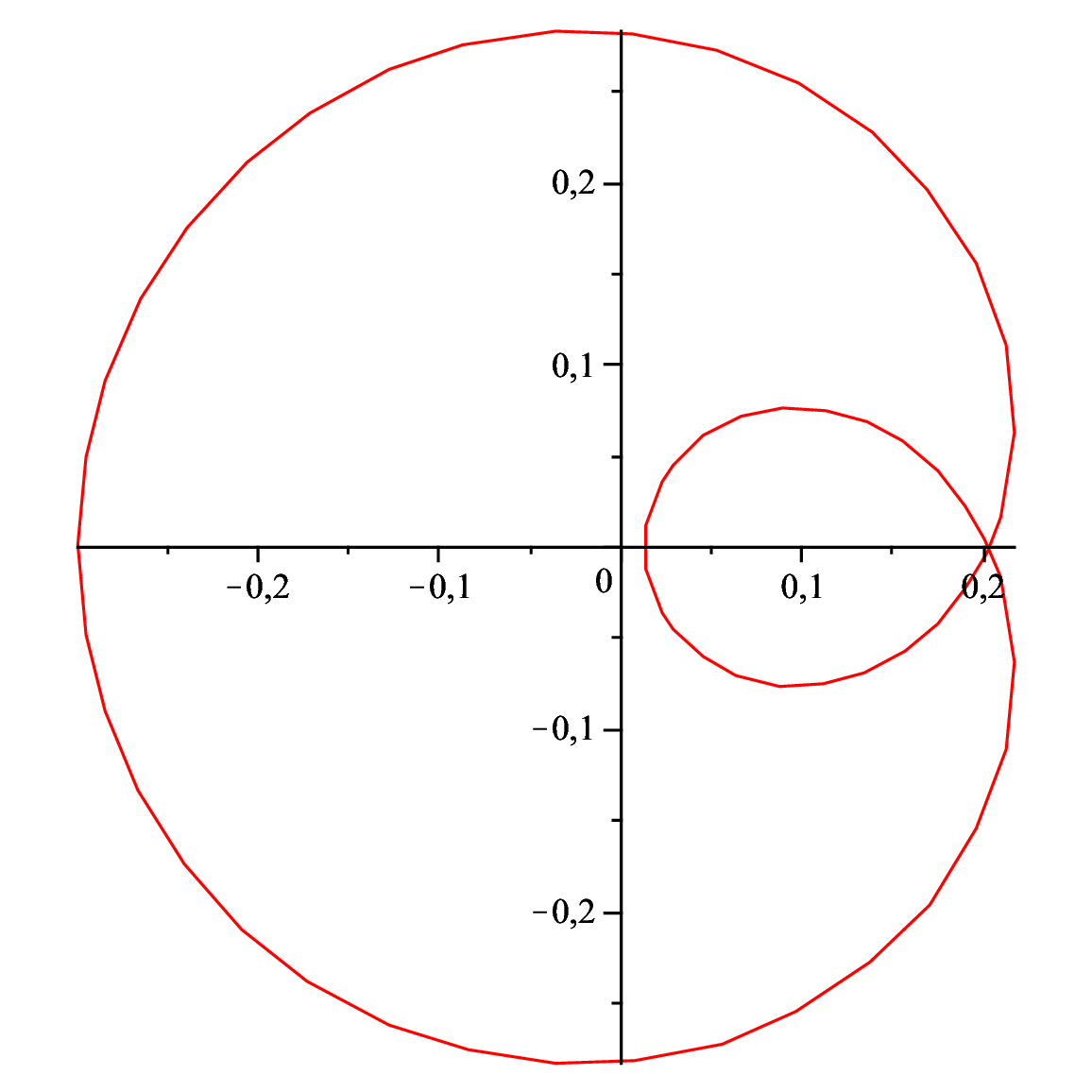}} \\
$\zeta=0.6$
\end{minipage}
\begin{minipage}[h]{0.3\linewidth}
\center{\includegraphics[width=0.7\linewidth]{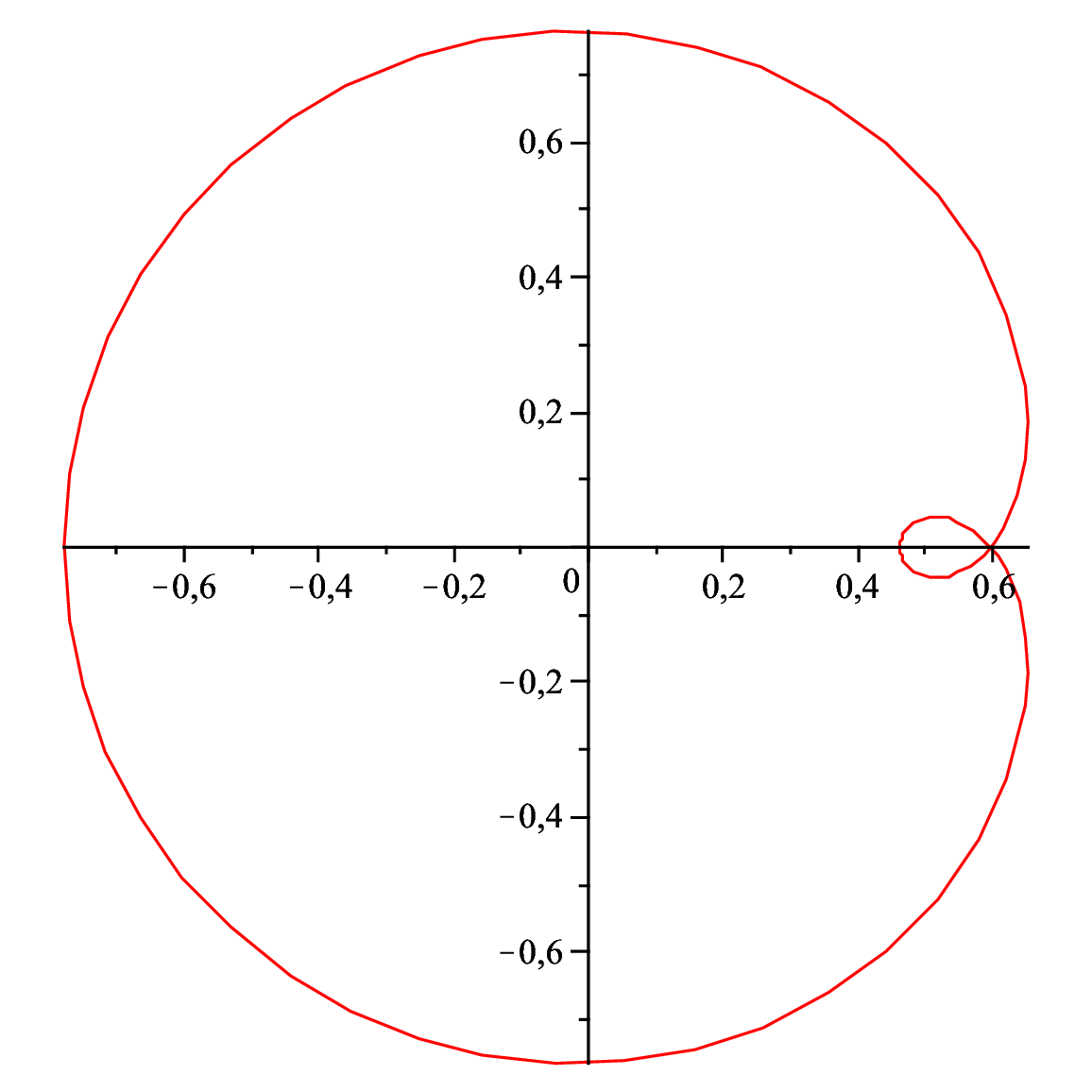}} \\
$\zeta=0.9$
\end{minipage}
\caption{self-intersecting solutions of \eqref{5.1}--\eqref{5.2}}
    \label{fig5.1}
\end{figure}

\medskip

Returning to the general case, we assume the operator ${\mathcal H}_\gamma$ to be represented by its matrix
$H_\gamma=(h_{nm})_{m,n\in{\Z}}$ in the trigonometric basis, see \eqref{5.7}. Doing numerics, we use the cut matrix
$H_\gamma^{(N)}=(h_{nm})_{m,n=-N}^N$ with a sufficiently large $N$.

Let $W$ is the maximal subspace of ${\mathbb C}^{2N+1}$ consisting of column-vectors
$$
\widehat\gamma=(\widehat\gamma_{-N},\dots,\widehat\gamma_{-1},0,\widehat\gamma_1,\dots,\widehat\gamma_N)^t
$$
satisfying
\begin{equation}
\|H_\gamma^{(N)}\widehat\gamma-\widehat\gamma\|<\varepsilon\|\widehat\gamma\|
                                \label{5.12}
\end{equation}
with a priory chosen small $\varepsilon>0$.
Here $\|x\|^2=\sum_{i=-N}^N|x_i|^2$ for $x\in{\mathbb C}^{2N+1}$.
The dimension $L=\mbox{dim}\,W$ is close to $N$ and $L=N$ if $N$ is sufficiently large and the value of $\varepsilon$ is appropriately chosen.
The space $W$ is considered as a finite-dimensional approximation of the space of solutions to the linear equation \eqref{5.2}.
Given the matrix $H_\gamma^{(N)}$, a basis of $W$ can be efficiently constructed.

As such a basis, it is convenient for us to take the set of vectors $\{e_1,... e_N\}\subset W$, where the positively-indexed coordinates of $e_k$ in $\mathbb{C}^{2N+1}$ are of the form
\begin{equation}
(0,\dots,0,1,0\dots 0), \ \text{1 is on the k-th position.}
\label{data}\end{equation}

Of course we replace the infinite system \eqref{5.5}--\eqref{5.6} with the finite system
\begin{equation}
\sum\limits_{m=-N}^Nm^2|\widehat\gamma_m|^2=1,
                                \label{5.13}
\end{equation}
\begin{equation}
\sum\limits_{m=-N}^N m(m+k)\widehat\gamma_m\overline{\widehat\gamma_{m+k}}=0\quad(k=1,2,\dots,(N-1))
                                \label{5.14}
\end{equation}
We look for solutions to the system \eqref{5.13}--\eqref{5.14} belonging to $W$
For the numerical solution of the system,
we use some iteration method (a version
of the classical gradient descent method). As in any iteration method, the choice of an initial approximation to the solution is of a
crucial importance.
Assuming the domain $\Omega$ to be sufficiently close to the unit disk, we choose the initial approximation
$\widehat\gamma{}^0\in W$  the vector $e_1$ of our basis.
%(the vector corresponds to the unit disk).
Applying our iteration algorithm, we obtain a sequence $\widehat\gamma{}^k\in W\ (k=0,1,\dots)$ converging to a solution of the system
\eqref{5.13}--\eqref{5.14}.
In most cases 10 iterations are enough to get a good approximation.

The method works successfully at least in examples presented in the next section.
Nevertheless, we emphasize that the system
\eqref{5.13}--\eqref{5.14} has many ``wrong solutions'' resulting functions $\gamma(s)$ which belong to ${\mathcal A}(\gamma)$ but do not
belong to ${\mathcal A}_s(\gamma)$. We got many such wrong solutions if either $\Omega$ was not close to the unit disk or
the initial approximation
$\widehat\gamma{}^0\in W$ was not sufficiently close to the vector $e_1$.
Three such ``wrong solutions'' are shown on Figures \ref{fig6.2} and \ref{fig6.4} below.

\section{Numerical examples}

In practical tomography, the matrix of the DtN-map is obtained by measurements. In our numerical experiments, we compute the matrix of the
operator $\Lambda_\Omega$ for a given compact domain $\Omega\subset{\R}^2$ bounded by a smooth curve $\Gamma$. This first part of the experiment
is referred to as {\it the forward problem}. We first shortly discuss the numerical solution of  the forward problem.

By the definition of the DtN-map, $\Lambda_\Omega f=\frac{\partial u}{\partial\nu}\big|_\Gamma$ for $f\in C^\infty(\Gamma)$, where $u$ is the
solution to the Dirichlet problem
\begin{equation}
\Delta u=0\ \mbox{in}\ \Omega,\quad u|_\Gamma=f.
                                \label{6.1}
\end{equation}

The most traditional method for numerical solution of the Dirichlet problem consists of replacing the differential equation $\Delta u=0$
with a system of difference equations on a grid adapted to the domain $\Omega$, and of solving the system of linear algebraic equations by some
iteration method. It is also possible to replace the Dirichlet problem with a second kind integral equation on the curve $\Gamma$ and to apply
well known methods of numerical solution of integral equations. In both approaches, we have to compute the derivative
$\frac{\partial u}{\partial\nu}\big|_\Gamma$. Numerical differentiation is rather unstable with respect to discretization errors. Thus,
the precise numerical solution of the forward problem is not easy.

In the case of a simply connected domain $\Omega$, the Dirichlet problem \eqref{6.1} is equivalent to the problem of constructing a biholomorhism
(i.e., one-to-one conformal map) $\Phi:{\D}\to\Omega$ ($\D$ is the unit disk).
Such a biholomorhism exists by the Riemann theorem.
Given $\Phi$, we define the diffeomorphism $\varphi=\Phi|_{\mathbb S}:{\mathbb S}=\partial{\D}\to\Gamma$ and define the function
$0<a\in C^\infty({\mathbb S})$ by
$
a(z)=\frac{1}{|\Phi'(z)|}\ (z\in{\mathbb S}),
$
where $\Phi'(z)$ is the complex derivative of the holomorphic function $\Phi(z)$. The operator
$\Lambda_\Omega$ is expressed through $\Lambda_{\D}$  by
\begin{equation}
\Lambda_\Omega=\varphi^{*-1}(a \Lambda_{\D})\varphi^*,
                              \label{6.2}
\end{equation}
where $a$ stands for the operator of multiplication by the function $a$. This relation follows immediately from the fact:
a conformal map transforms harmonic functions again to harmonic functions, see details in \cite{JS0}. The operator $\Lambda_{\D}$
is well known: $\Lambda_{\D}e^{in\theta}=|n|e^{in\theta}\ (n\in{\Z})$. Together with the latter equality, the formula \eqref{6.2}
allows us easily compute the matrix of $\Lambda_\Omega$. It is important to note that the problem of numerical differentiation (computation
of $\frac{\partial u}{\partial\nu}\big|_\Gamma$) is now replaces with computing the complex derivative $\Phi'(z)$. In particular, if $\Phi$
is given by an explicit analytic formula, the derivative $\Phi'(z)$ can be obtained analytically without any numerics.

In some partial cases, like in Example 1 below, the biholomorhism $\Phi:{\D}\to\Omega$ is found in Complex Analysis. But for a general
simply connected domain $\Omega$, the construction of the biholomorhism is the serious numerical problem actually equivalent to the Dirichlet
problem \eqref{6.1}.

In Examples 2--4 below, we avoid the numerical solution of the Dirichlet problem with the help of the following trick.

We start with choosing a holomorphic function $\widetilde\Phi(z)$ that is injective on $\D$, i.e.,
$\widetilde\Phi(z_1)=\widetilde\Phi(z_2)$ implies $z_1=z_2$ for $z_1,z_2\in\D$. Then
$\widetilde\Phi|_{\D}$ is a biholomorphism of the unit disk onto some domain. Then we normalize this biholomorphism by setting
$\Phi=c\widetilde\Phi$, where the constant $c>0$ is chosen so that the length of the boundary curve $\Gamma$ is equal to $2\pi$,
and set $\Omega=\Phi({\D})$. Since the functions $\Phi(z)$ and $\Phi'(z)$ are given, we compute the matrix of $\Lambda_\Omega$ %by \eqref{6.2}
without any numerics. In examples 2--4 below, $\Phi(z)$ is a low degree polynomial.

Finally we recall that matrices of the operators $\Lambda_\gamma$ and ${\mathcal H}_\gamma$ are related by \eqref{5.7a}. Just the matrix
$H^{(N)}_\gamma=(h_{mn})_{m,n=-N}^N$ of ${\mathcal H}_\gamma$ is used in our algorithm for solving the inverse problem.\vskip2mm

{\bf Example 1. Ellipse.}
Let $\Omega$ be the domain bounded by the ellipse
$
\frac{x^2}{a^2}+\frac{y^2}{b^2}=1.
$
Constants $0<b<a$ are chosen so that the length of the ellipse is equal to $2\pi$. We still have the free parameter $a/b$.
The biholomorphism of the unit disk onto the domain bounded by the ellipse is expressed in terms of second kind elliptic integrals
that are tabulated in many softwares.
This allows us to compute the matrix of the DtN-map with a good precision at least if $a/b$ is not large.

In the case of $a/b=13/12$, the central $3\times3$-part of the matrix $H^{(N)}_\gamma$ looks as follows with precision $10^{-2}$:
$$
\left(\begin{array}{ccc}
   h_{-1,-1} &  h_{-1,0} &  h_{-1,1}  \\
   h_{0,-1} & h_{0,0} & h_{0,1}   \\
   h_{1,-1} & h_{1,0} &  h_{1,1}
\end{array}\right)
=\left(\begin{array}{ccc}
   - 1 & 0 & 0.06  \\
  0 & 0 & 0   \\
    -0.06 & 0 &  1
\end{array}\right).
$$
Other elements coincide with corresponding elements of the matrix \eqref{H_circle} with the same precision.
Taking the symmetry \eqref{5.7b} into account, we see that the matrix differs from the case of the unit disk by one element $0.06$.
Nevertheless, the difference is sufficient for the good reconstruction of the ellipse.
Practically we used the matrix $H^{(10)}_\gamma$ with precision $10^{-5}$.

In the case of $a/b=3/2$, the matrix $H^{(3)}_\gamma$ is
$${\small{
\left(\begin{array}{rrrrrrr}
  - 1 & 0 & 0 & 0 &  0.01 & 0 &  0.01 \\
 0 & - 1 & 0 & 0 & 0 &  0.06 & 0 \\
 0& 0 & - 1.05 & 0 &  0.31 & 0 &  0.03 \\
 0 & 0 & 0 & 0 & 0 & 0 & 0 \\
-0.03 & 0 & -0.31 & 0 &  1.05 & 0 & 0 \\
 0 & -0.06 & 0 & 0 & 0 &  1 & 0 \\
-0.01 & 0 & -0.01 & 0 &  0 & 0 &  1
\end{array}\right).
}}$$
In both the cases, the reconstructed ellipse is visually indistinguishable from the original one. However, we are not able to compute the matrix of the DtN-map for $a/b>5$ with a sufficiently good precision
because of some difficulties concerning computations with second kind elliptic integrals. Th \href{https://drive.google.com/file/d/1yIKpn4lKWDaUO_JX6FMJpz31uNvgFRRK/view?usp=sharing}{animation for ellipse}, available at link, see also \cite{url}  shows how the algorithm's iterations first linger for a while on a false solution and then converge to the true one. The black line represents the velocity vector.

{\bf Example 2. Disk with two symmetric dents.}
The domain $\Omega$ is defined by
$\Omega=c\{1.5z+0.4z^3\mid |z|\le1\}$ with the constant $c$ chosen so that the length of the boundary curve is equal to $2\pi$.
In other words, $\widetilde\Phi(z)=1.5z+0.4z^3$ for the function mentioned at the beginning of the current section.
The original curve is drawn on the left picture of Figure \ref{fig6.1}. Two reconstructed curves for $N=5$ (red) and for $N=20$ (blue) are
presented on the middle picture.
Observe that the reconstructed curve has a shallower dent
than the original, making it appear longer. This is a reminiscent of the problem: how would a string wrapped around the Earth equator
rise if the string was extended by 10 cm?
The right picture of Figure \ref{fig6.1} shoes the velocity vector, i.e., the curve
$s\mapsto\frac{d\gamma(s)}{ds}$, where $s$ is the arc length.
The loops formed by the velocity vector correspond to zones where the curvature of $\gamma$ changes its sign.

 \begin{figure}[h]
 \begin{minipage}[h]{0.37\linewidth}
 \center\includegraphics[width=0.9\linewidth]{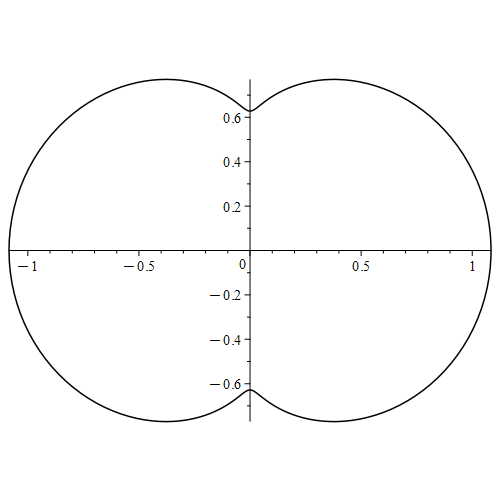}
 \end{minipage}
 %\begin{minipage}[h]{0.3\linewidth}
 %\center\includegraphics[width=0.7\linewidth]{iteration 1234.png} \\
 %\end{minipage}
 \begin{minipage}[h]{0.37\linewidth}
 \center\includegraphics[width=0.97\linewidth]{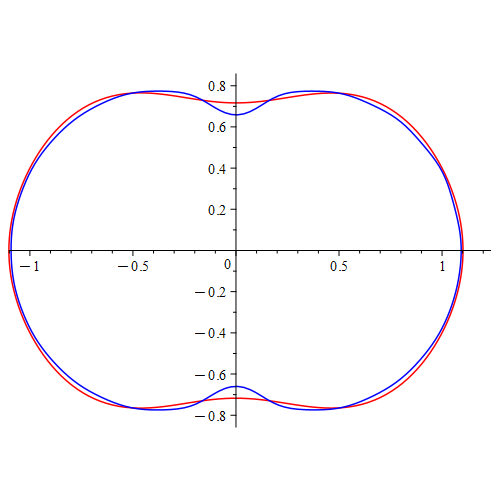}
 \end{minipage}
 \begin{minipage}[h]{0.23\linewidth}
 \center\includegraphics[width=0.9
 \linewidth]{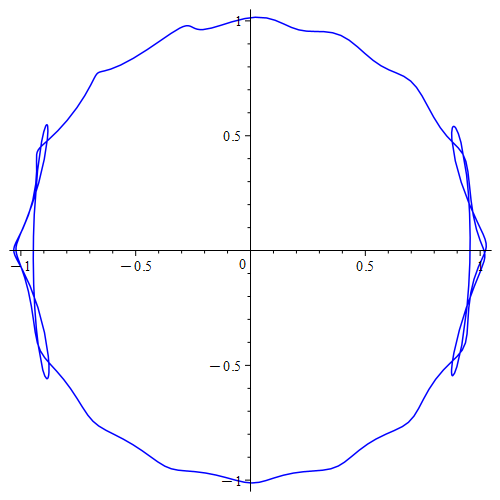}
 \end{minipage}
 \caption{Example 2: the original curve, two reconstructed curves for $N=5$ and for $N=20$, the velocity vector for $N=20$}
 \label{fig6.1}
 \end{figure}

{\bf Example 3. Disk with a dent}.
$\Omega=c\{(z+1.5)^2\mid |z|\le1\}$. Here we used the matrix $H^{(10)}_\gamma$. The reconstructed curve obtained after one(!) iteration of our
algorithm is shown on the left picture  of Figure \ref{fig6.2}, it is indistinguishable from the original curve at the scale of the picture.
Two ``wrong solutions'' mentioned at the end of the previous section are shown on two right
pictures of Figure \ref{fig6.2}.
\begin{figure}[h]
\begin{minipage}[h]{0.3\linewidth}
\center\includegraphics[width=0.9\linewidth]{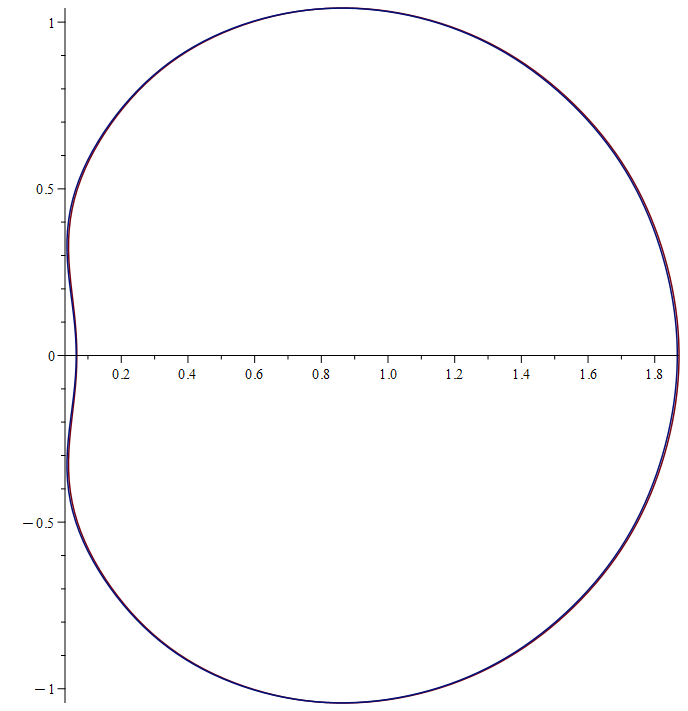} \\
\end{minipage}
\begin{minipage}[h]{0.3\linewidth}
\center\includegraphics[width=0.9\linewidth]{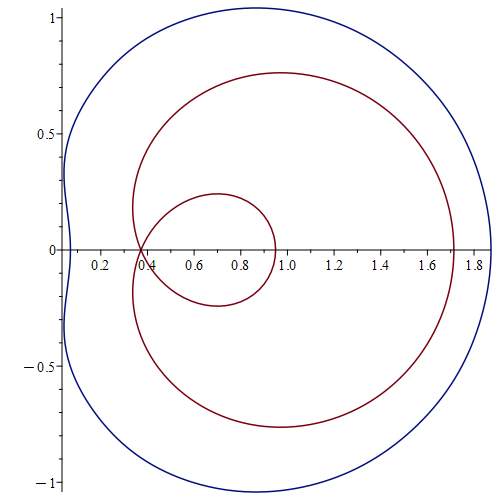} \\
\end{minipage}
\begin{minipage}[h]{0.3\linewidth}
\center\includegraphics[width=0.9\linewidth]{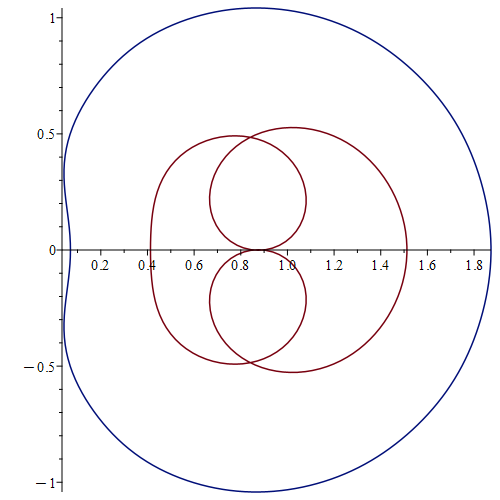} \\
\end{minipage}
\caption{Example 3: disk with a dent, the recovered and reconstructed curves, two ``wrong solutions''}
\label{fig6.2}
\end{figure}

%For clarity, we use this example to demonstrate how the successive approximations defined by our algorithm behave when the initial
%data are not too close to \eqref{data}. We chose a pair of initial data such that the algorithm's iterations "separate" them into
%two different solutions, see Fig.\ref{fig6.66}.    Initial approximation $e^W_1+e^W_2+1.62e^W_3$ corresponds to black curve
%("true solution"), $e^W_1+e^W_2+1.64e^W_3$ corresponds to red curve ("wrong solution"). See the notation at the end of section 5).
%Note in connection with Theorem 4 that the "wrong solutions" on the left side of Figures \ref{fig6.2} and \ref{fig6.66}
%have the same rotation number equal to $6\pi$.

Figure \ref{fig6.4} illustrates our iteration algorithm of numerical solution of the system \eqref{5.13}--\eqref{5.14} for Example 3.
The result depends on the choice of the initial approximation ${\widehat\gamma}^{0}\in W$.
Let $e_k$  be as in \eqref{data}.
Two slightly different initial approximations are drawn on the left picture:
${\widehat\gamma}^{0}=e_1+e_2+1.62e_3$ (black curve) and ${\widehat\gamma}^{0}=e_1+e_2+1.64e_3$ (red curve). Three last pictures of Figure \ref{fig6.4} show curves
obtained by three iterations of the algorithm. Black curves converge to the ``true solution'' while red curves converge to a ``wrong solution''
that belongs to ${\mathcal A}(\gamma)$ but does not belong to ${\mathcal A}_s(\gamma)$.

\begin{figure}[h]
\begin{minipage}[h]{0.24\linewidth}
%\center{\includegraphics[width=0.98\linewidth]{f1.png}}
%\end{minipage}
%\begin{minipage}[h]{0.2\linewidth}
\center{\includegraphics[width=0.98\linewidth]{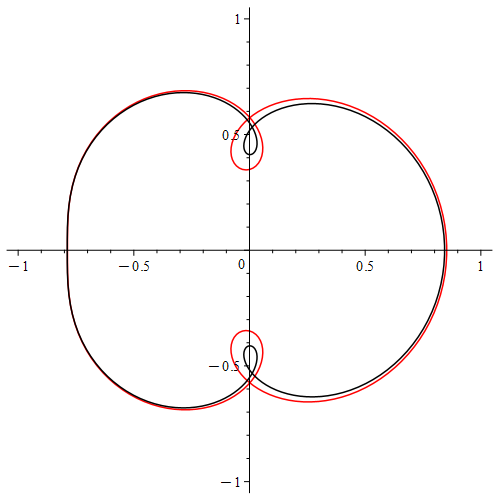}}
\end{minipage}
\begin{minipage}[h]{0.24\linewidth}
\center{\includegraphics[width=0.98\linewidth]{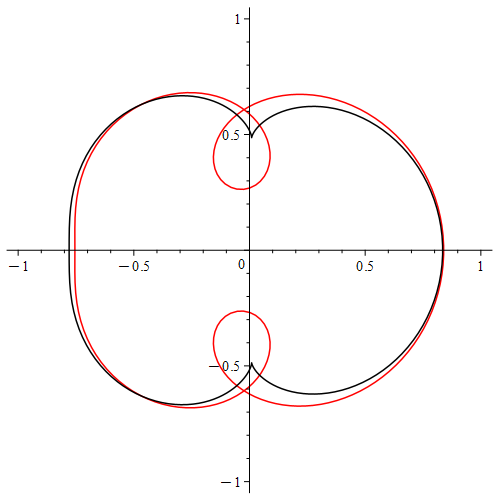}}
\end{minipage}
\begin{minipage}[h]{0.24\linewidth}
\center{\includegraphics[width=0.98\linewidth]{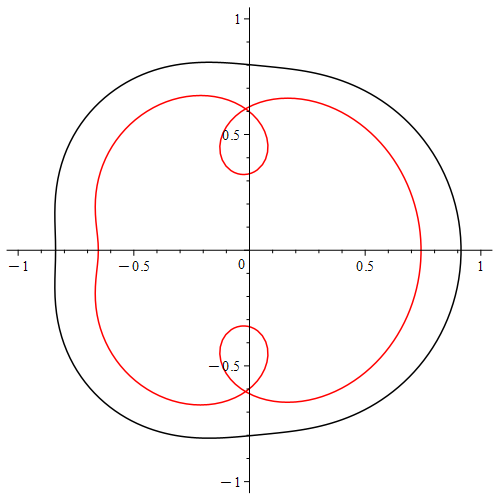}}
\end{minipage}
\begin{minipage}[h]{0.24\linewidth}
\center\includegraphics[width=0.98\linewidth]{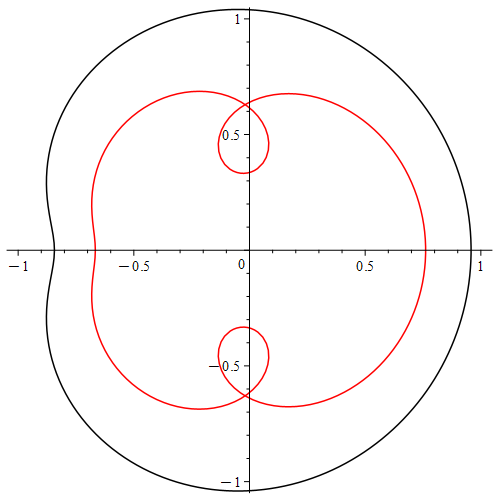}
\end{minipage}
\caption{Example 3: two initial approximations, curves obtained by three iterations}
    \label{fig6.4}
\end{figure}

%In this case, the matrix $H_\gamma$ will have complex coefficients.
%Figure 7 on the left shows the curve $c(6e^{it} + (0.5-0.7i)e^{2it} + e^{4it})$. In the middle part of Figure 7, the solid line shows the
%approximation of the Fourier series expansion for the natural parameter, $-10\leq n\leq 10$, and the dashed line shows the reconstructed
%curve for $N=10$. The right side of Figure 7 shows the velocity vector of the reconstructed curve. We see that it corresponds well
%to the naturalness of the parameter. Thus, we recognize not only the configuration pattern but also the phase pattern.
%Notice how the velocity vector makes loops in places where the curvature of the curve changes sign.
%%
%The animation (available in the electronic version of the journal) shows how similar initial data diverge to different solutions of
%the equation \eqref{60.17+0.6} during the iterations of the algorithm.

\begin{figure}[h]
\begin{minipage}[h]{0.3\linewidth}
\center{\includegraphics[width=0.98\linewidth]{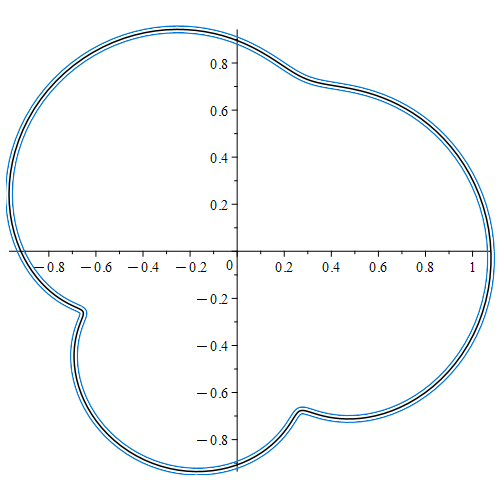}}
\end{minipage}
\begin{minipage}[h]{0.3\linewidth}
\center{\includegraphics[width=0.98\linewidth]{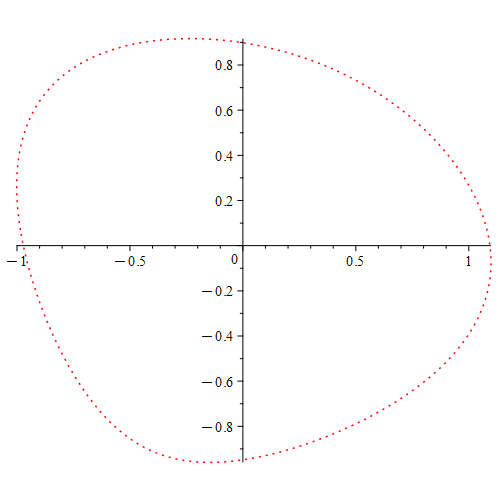}}
\end{minipage}
\begin{minipage}[h]{0.3\linewidth}
\center{\includegraphics[width=0.98\linewidth]{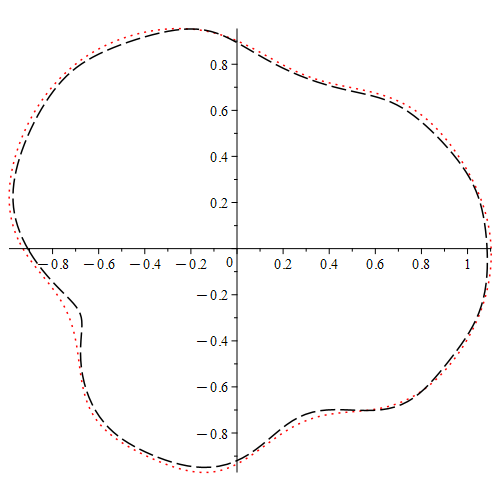}}
\end{minipage}
%\begin{minipage}[h]{0.3\linewidth}
%\center\includegraphics[width=0.98\linewidth]{asymm_diff.png}
%\end{minipage}
\caption{ Example 4: the original curve, the reconstructed curve for $N=3$, and for $N=10$}
    \label{fig6.5}
\end{figure}
%\begin{figure}[h]
%\begin{minipage}[h]{0.24\linewidth}
%\center{\includegraphics[width=0.98\linewidth]{15.png}}
%\end{minipage}
%\begin{minipage}[h]{0.24\linewidth}
%\center{\includegraphics[width=0.98\linewidth]{16.png}}
%\end{minipage}
%\begin{minipage}[h]{0.24\linewidth}
%\center{\includegraphics[width=0.98\linewidth]{17.png}}
%\end{minipage}
%\begin{minipage}[h]{0.24\linewidth}
%\center\includegraphics[width=0.98\linewidth]{18.png}
%\end{minipage}
%\begin{minipage}[h]{0.24\linewidth}
%   \center{\includegraphics[width=0.98\linewidth]{19.png}}
%\end{minipage}
%\begin{minipage}[h]{0.24\linewidth}
%\center{\includegraphics[width=0.98\linewidth]{20.png}}
%\end{minipage}
%\begin{minipage}[h]{0.24\linewidth}
%\center\includegraphics[width=0.98\linewidth]{21.png}
%\end{minipage}
%\begin{minipage}[h]{0.24\linewidth}
%\center\includegraphics[width=0.98\linewidth]{22.png}
%\end{minipage}
%\caption{ some of iterations of recovery: close initial data diverge }
%    \label{fig6.6}
%\end{figure}

{\bf Example 4. Asymmetric curve}. $\Omega=c\{6z+(0.5-0.7i)z^2+z^4\mid|z|\le1\}$.
Unlike previous examples, the curve $\Gamma=\partial\Omega$ has no symmetry axis, see the remark after \eqref{5.7b}.
Therefore $H^{(N)}_\gamma$ is a complex-valued matrix. In particular, $H^{(3)}_\gamma=A+iB$ where $A$ and $B$ are
$$\footnotesize{
\left(\begin{array}{rrrrrrr}
- 1.01 &  0.01 &  0.00  & 0 & - 0.01 & - 0.08  &  0.00\\
  0.01  & - 1.03  &  0.00 & 0 & 0.12  & -0.01  & -0.12\\
  0.01 &  0.00 & - 1.04  & 0 & 0.08 &  0.23  & -0.02\\
 0 & 0 & 0 & 0 & 0 & 0 & 0\\
  0.02 &  -0.23 &  -0.08  & 0 &  1.04 &  0.00 &  -0.01\\
 0.12 &  0.01& -0.12 & 0 &  0.00 &  1.03 & -0.01\\
  0.00 & 0.08 & 0.01  & 0 & 0.00  & 0.01 & 1.01
\end{array}\right),
\left(\begin{array}{rrrrrrr}
0.00 &  -0.01 & 0.00  & 0 &  0.04   &  -0.03   & 0.01\\
0.00 & 0.00  &  0.00 & 0 &  -0.05  &   0.09  &  -0.05\\
-0.01 &  0.00 & 0.00  & 0 &  0.07 & -0.10  & 0.12\\
 0 & 0 & 0 & 0 & 0 & 0 & 0\\
  0.12 & -0.10 &  0.07  & 0 &  0.00 &  0.00 &  -0.01\\
- 0.05  &  0.09 & - 0.06  & 0 &  0.00 &  0.00 & 0.02\\
0.01 & -0.03 & 0.04 & 0 & 0.00 & -0.01 & 0.00
\end{array}\right).}
$$

The original curve is drawn on the left part of Figure \ref{fig6.5}. The curve reconstructed from the matrix $H^{(3)}_\gamma$ is shown on
the middle part of Figure \ref{fig6.5}. There are two lines on the right part of Figure \ref{fig6.5}: the red dotted line  is the
curve reconstructed from the matrix $H^{(10)}_\gamma$; the black dash line represents the partial sum of the Fourier series for the original
curve
\begin{equation}
\gamma(s)\approx\sum\limits_{m=-10}^{10} \widehat\gamma_m\,e^{ims},
                              \label{6.3}
\end{equation}
where $s$ is the arc length (compare with \eqref{5.4}).

%Let's also present the matrix $H^*$ for Example 5. The last column was not included, but it can be reconstructed from the first column,
%since $h_{-i,-j}=-\overline{h_{i,j}}$. To get an idea of the reconstruction accuracy, we present the corresponding curve reconstructed
%from this matrix ($N=3$).Although the corresponding trigonometric polynomial is only of the third order, “flattening” at the site
%of future depressions is already visible

%\begin{figure}
%    \centering
%    \includegraphics[width=0.3
%    \linewidth]{little_asymm.png}
%    \caption{Example 5, recover with $N=3$}
%    \label{}
%   \end{figure}

\medskip

All presented examples confirm our main thesis: {\it The Dirichlet-to-Neumann map is very sensitive to small deviations of the shape of a
domain}.

\end{document}